\documentclass[letterpaper, 10 pt, conference]{ieeeconf} 
\IEEEoverridecommandlockouts                                                                                   
\usepackage{graphics} 
\usepackage{epsfig} 
\usepackage{amsmath}
\usepackage{amsfonts} 
\usepackage{xcolor}
\usepackage{balance}
\usepackage{dblfloatfix}
\usepackage{soul}
\usepackage{latexsym,theorem}
\usepackage{cite}
{\theorembodyfont{\slshape}\newtheorem{theorem}{Theorem}[section]}
{\theorembodyfont{\slshape}}
{\theorembodyfont{\slshape}\newtheorem{lemma}[theorem]{Lemma}}
{\theorembodyfont{\slshape}}
{\theorembodyfont{\slshape}}
{\theorembodyfont{\upshape}}
{\theorembodyfont{\upshape}\newtheorem{problem}[theorem]{Problem}}
{\theorembodyfont{\upshape}\newtheorem{remark}[theorem]{Remark}}
{\theorembodyfont{\upshape}\newtheorem{assumption}[theorem]{Assumption}}
{\theorembodyfont{\upshape}}
\title{\LARGE \bf
Dynamic Output-Feedback Controller Synthesis for Dissipativity and $H_2$ Performance from Noisy Input-Output Data
}
\author{Pietro Kristović$^{1}$, Andrej Jokić$^{1}$ and Mircea Lazar$^{2}$
\thanks{$^{1}$ Faculty of Mechanical Engineering and Naval Architecture, University of Zagreb, Zagreb, Croatia, {\tt\small pietro.kristovic@fsb.unizg.hr, andrej.jokic@fsb.unizg.hr}}
\thanks{$^{2}$ Electrical Engineering Faculty, Eindhoven University of Technology, Eindhoven, The Netherlands
        {\tt\small m.lazar@tue.nl}}
\thanks{This research has been supported by the European Regional Development Fund under grant agreement PK.1.1.10.0007 (DATACROSS).}
}
\begin{document}
\maketitle
\thispagestyle{empty}
\pagestyle{empty}
\begin{abstract}
In this paper we propose dynamic output-feedback controller synthesis methods  for discrete-time linear time-invariant systems. 
The synthesis goal is either to achieve dissipativity with respect to a given quadratic supply rate, or to achieve given $H_2$ performance level.  
It is assumed that the autoregressive model of system dynamics is unknown,
expect for the noisy disturbance term which is not part of the performance channel.
Instead, we have a recorded trajectory of inputs and outputs 
  which can be corrupted by an unknown but bounded disturbance.
Methods are formulated  in terms of  linear matrix inequalities parametrized by a scalar variable, while in noiseless case they reduce to linear matrix inequalities.
 Within the considered setting, synthesis procedures are non-conservative.

\end{abstract}
\section{INTRODUCTION}
\label{IntroductionLabel}

In this paper, we propose dynamic output-feedback controller synthesis methods for discrete-time linear time-invariant (LTI) systems. More precisely, we consider an LTI system which admits a representation in the following autoregressive (AR) form
\begin{align} 
&y(t)+A_1 y(t-1) + \ldots + A_l y(t-l) = \nonumber \\
&B_{u0} u(t)+B_{u1} u(t-1)+\ldots + B_{ul} u(t-l)+ \nonumber \\
&B_{w0} w(t)+B_{w1} w(t-1)+\ldots + B_{wl} w(t-l)+ \nonumber \\
& + B_{d0} d(t) \label{IntroSys} 
\end{align}
where $y$ is the measured output for control purposes, $u$ is the control input, $w$ is the disturbance input and $d$ is the noise. Additionally, we add the output $z$ to the system, which is defined as a linear combination of
\[ \{ y(t-1),\ldots,y(t-l),u(t),\ldots,u(t-l),w(t), \ldots,w(t-l) \}.\] 
The channel $w \,\, \rightarrow \,\, z$ is the performance channel and we are concerned with the following controller synthesis problem related to that input-output pair: (\emph{a}) render the closed-loop system dissipative with respect to a given generic unstructured quadratic supply rate, (\emph{b}) achieve a given $H_2$ performance level. Indeed, stability of the closed-loop system is implicitly included.

The results presented in this paper belong to the direct data-driven control framework. In its approach to the controller synthesis, this paper adds to the line of research in the \textit{informativity} framework, which has gained significant attention during the past years, see,
e.g., \cite{R27} for an overview.
We assume no knowledge of the system matrices in \eqref{IntroSys}, except for the matrix $B_{d0}$. We also assume that we know how the performance output signal $z$ is formed from $y$, $u$ and $w$. To compensate for the unknown part of the model, we assume to have recorded system's trajectory\footnote{Detailed description of the problem setting is presented in the second section of the paper.}, which is possibly corrupted by an unknown but bounded noise $d$ (see \eqref{IntroSys}).   

Data-driven methods for generic LTI systems with available state measurements, which rely on data from input-state trajectories, have been presented in \cite{R26, c4, c6, R5564636, R24, zz1, zz66, c5}. All these works are concerned with static state-feedback controller synthesis, while the synthesis objectives include stabilization, $H_2$, $H_\infty$ and generic quadratic performance.
In contrast, the input-state data was used to derive \emph{dynamic output-feedback} controller synthesis methods for dissipativity, $H_2$ and $H_\infty$ performance in \cite{66999} and \cite{999}.
In summary, all these methods use input-state data. 

It has been recognized that AR models can be suitably rewritten into input-state form and the static state-feedback synthesis methods can be adopted to synthesize output-feedback controllers for such systems. This has been done in   \cite{R26}, \cite{c5}, \cite{a1}, \cite{zz2}, where again stabilization, $H_2$, $H_\infty$ and quadratic performance have been considered, while a behavioural framework is used in \cite{R23} for stabilization.
 However, one specific feature of all these results is that the disturbance  affects the system in a specific way, with no associated dynamics. 
With reference to \eqref{IntroSys}, 
there is no input $d$ and the input $w$ affects the system only through known matrix $B_{w0}$ while $B_{w1}, \ldots, B_{wl}$ are not present. That is what we mean by the term that no dynamics is associated with the disturbance. Therefore, the considered systems are of a specific AR class.

It is furthermore important to emphasize that all of the above mentioned results are applicable only under the specific \textit{restriction} that  $n=pl$, where $p$ is the number of outputs available for control (in reference to \eqref{IntroSys} there is $y(t)\in \mathbb{R}^p$), $l$ is the system's lag (see \eqref{IntroSys}) while $n$ is the order of a minimal state space realization of the system.
This restriction can be resolved by constructing special non-minimal realization using the data as shown in \cite{R33}. Based on this, and in case the signal-to-noise ratio is sufficiently large, robust
controller synthesis methods for stabilization from \cite{c6} and
\cite{R5564636} can be used. However, in \cite{R33} it remains unclear how to define and to interpret disturbance term in the model of system dynamics.
The restriction $n=pl$ is also resolved in \cite{R22}, where the system stabilization based on data corrupted by a measurement error is considered.
This is achieved by  constructing an auxiliary
system which extends system realization. However, it is not clear how to construct this auxiliary system, as stated in  \cite[Sec. VI]{R33}. 
Furthermore, in both \cite{R33} and \cite{R22} the structure of the constructed system realization is not exploited for the purpose of controller synthesis what implies that the obtained solutions are in general conservative. These issues regarding controller synthesis, but again for specific AR models with no dynamics associated with the disturbace, have been resolved in \cite{R42667} where dissipativity and $H_2$ performance are considered.

The main contributions of this paper are summarized as follows. We consider a generic LTI system that can be represented in the form of AR model \eqref{IntroSys} and provide a data-driven dynamic output-feedback controller synthesis method for dissipativity and $H_2$ performance of the closed-loop system.
The natural state-space realization of a generic AR model
is non-minimal where $n \leq pl$ in general, thus, the state data matrix may not have a full row rank, which 
prevents direct application of existing state-feedback methods.
 The key idea here is to project the non-minimal realization onto the reachable subspace identified by the data (via compact SVD of the state data matrix), obtaining a controllable reduced realization whose dimension equals the rank of the data. The AR shift structure is then exploited  by separating the unknown system matrices from known structural parts, enabling a non-conservative application of the matrix S-lemma over the data-consistency set.
The obtained synthesis conditions are in a form of linear matrix inequalities (LMIs) parametrized by a scalar variable, thus, the controller parameters can be obtained by solving the LMIs multiple times in a line-search procedure.
In the noiseless case, the synthesis conditions reduce to LMIs.
Obtained synthesis methods are non-conservative in the sense that they are necessary and sufficient conditions
for a set of systems that are consistent with the recorded data and the known disturbance bound.
To illustrate the results, we use an academic numerical example from \cite{R22} and an example of active car suspension, taken from \cite{MatlabRef}.

\subsubsection*{Notation} 
We use $\mathbb{R}$, $\mathbb{N}_{\geq 0}$, $\mathbb{R}^n$ and  $\mathbb{R}^{m\times n }$ to denote a field of real numbers, a set of natural numbers, an $n$-dimensional vectors with elements in $\mathbb{R}$, and $m$ by $n$ matrices with elements in $\mathbb{R}$, respectively.
 We use $M^\dagger $ to denote the Moore-Penrose pseudo-inverse
of a real matrix $M$, while $\text{im}(M)$, $\text{ker}(M)$, $\text{tr}(M)$ and $\text{rank}(M)$ denote the image, kernel space, trace and rank of matrix $M$, respectively.
With matrices
$A_1,...,A_n$ which have the same number of columns, we
use $\text{col}(A_1,..., A_n)$ to denote the matrix $\begin{pmatrix} A_1^\top \cdots A_n^\top \end{pmatrix}^\top$. 
For a real square matrix $M$ we use $\text{He}\{M\}$ to denote a symmetric matrix $M^\top + M$. We use $\text{diag}(A,B)$ to denote the matrix \begin{small}$\begin{pmatrix} A & 0 \\ 0 & B  \end{pmatrix}$\end{small}. 
In a linear matrix inequality (LMI), $\star$ represents blocks which can be inferred from symmetry.  
When replacing a set of inequalities $C \prec 0$, $A-BC^{-1}B^\top \prec 0$ with a matrix inequality of the form \begin{small}$\begin{pmatrix} A & B \\ B^\top & C \end{pmatrix}\prec 0$\end{small} (and vice versa), we will say that we have \emph{used the Schur complement with respect to the matrix $C$}. In such a way we explicitly specify which block diagonal matrix is inverted. When we say that \textit{we apply a congruence transformation on a matrix inequality $A\succ 0$ with respect to the matrix $S$}, we mean that $S$ is a full rank matrix, and that the transformed inequality has the form $S^\top A S \succ 0$. 
 With $I_m$ and $0_{p_1 \times p_2}$ we denote the identity matrix of the size $m$, and the zero matrix of size $p_1 \times p_2$, respectively.
 We use  $\mathcal{C}(A,B)$ to denote the corresponding controllability matrix, where $A$ and $B$ are the state and input matrix, respectively.
\section{Preliminaries}
\subsection{Dissipativity}
\label{SecDissipativity}
Consider a discrete time LTI system 
\begin{equation}
\label{Sys1}
\begin{pmatrix} x(t+1) \\ z(t) \end{pmatrix} =
\begin{pmatrix} A & B \\ C & D \end{pmatrix}
\begin{pmatrix} x(t) \\  w(t) \end{pmatrix} 
\end{equation}
where $t \in \mathbb{N}_{\geq 0}$, $x(t) \in \mathbb{R}^n$, $w(t) \in \mathbb{R}^{m_w}$ and $z(t) \in \mathbb{R}^{p_z}$ represent time, the system state, input and output, respectively. We say that \eqref{Sys1} is strictly dissipative with respect to a quadratic supply function 
\begin{equation}
\label{supplyFunction}
s(w(t),z(t)):=
\begin{pmatrix}
w(t) \\ z(t)
\end{pmatrix}^\top
\begin{pmatrix}
-Q & -S \\
-S^\top & -R
\end{pmatrix}
\begin{pmatrix}
w(t) \\ z(t)
\end{pmatrix}
\end{equation}
 if there exist a symmetric matrix $P$ such that
\begin{equation}
\label{Dissip2}
\begin{pmatrix}
I & 0 \\
A & B \\
0 & I\\
C & D  \\
\end{pmatrix}^\top
\begin{pmatrix}
-P & 0 & 0 & 0  \\
0& P & 0 & 0   \\
0 & 0 & Q & S \\
0 & 0 & S^\top & R\\
\end{pmatrix}
\begin{pmatrix}
I & 0 \\
A & B \\
0 & I\\
C & D  \\
\end{pmatrix}
\prec 0, 
\end{equation}
in which case $V(x(t))=x(t)^\top P x(t)$ acts as a storage function.
If in addition we have $P\succ 0$ and $R\succeq 0$, then \eqref{Dissip2} implies stability of system \eqref{Sys1} since the Lyapunov inequality $A^\top P A-P\prec 0$ is incorporated in \eqref{Dissip2}.
\begin{remark}
\label{RemarkPrvi1}
The channel $w  \rightarrow z $ achieves $H_\infty$ performance of at least $\gamma $ (with $\gamma >0$) if and only if \eqref{Dissip2} holds for
$Q=-\gamma^2 I_{m_w}$, $R = I_{p_z}$,  $S = 0$ and some $P \succ 0$.
\end{remark}
\subsection{$H_2$ performance}
\label{H2performanceSubsection}
Consider a discrete time LTI system \eqref{Sys1}
and a $H_2$ norm on the performance channel $w \rightarrow z$ denoted by $||T(e^{j\omega})||_2$, where
\begin{equation}
\label{PrijenosnaFunkcija}
T(e^{j\omega} )=C (e^{j\omega}I-A)^{-1}B+D.
\end{equation}
According to \cite{zz00}, $||T(e^{j\omega})||_{2} < \mu $ if and only if  $\text{tr}(Z)< \mu^2$ and
\begin{equation}
\label{H2LMI}
\begin{pmatrix}
P-A P A^\top & B \\
B^\top &  I
\end{pmatrix}
 \succ 0, 
 \quad 
 \begin{pmatrix}
Z-DD^\top  & CP \\
P  C^\top &  P
\end{pmatrix} \succ 0,
\end{equation}
in which case the channel $w \rightarrow z $ achives a $H_2$ performance of at least $\mu$. Furthermore,
\eqref{H2LMI} implies stability of system \eqref{Sys1} since Lyapunov inequalities are incorporated in \eqref{H2LMI}.
\subsection{Matrix S-lemma}
The following version of the matrix S-lemma has been presented in \cite[Thm. 4.10]{c21}. 
\begin{theorem}
\label{Theorem0}
Let $M, H \in \mathbb{R}^{(n+r)\times (n+r)}$ be symmetric matrices with partitions \begin{small}$H:=\begin{pmatrix} H_{11} & H_{12} \\ H_{12}^\top & H_{22} \end{pmatrix}$\end{small} where $H_{11}\in \mathbb{R}^{n \times n}$,
and let the set $S_{H}$ be defined as
\begin{equation}
S_{H}:=\{ Z\in\mathbb{R}^{r\times n} | \begin{pmatrix} I \\ Z \end{pmatrix}^\top H \begin{pmatrix} I \\ Z \end{pmatrix} \succeq 0 \}.
\end{equation}
Assume that $H_{22}\prec 0$ and $S_{H} \neq \emptyset$.
Then, we have that
\begin{equation}
\begin{pmatrix}
I \\
Z
\end{pmatrix}^\top
M
\begin{pmatrix}
I \\
Z
\end{pmatrix}\succ 0 \text{ for all } Z\in S_{H}
\end{equation}
if and only if there exists scalar $\alpha \geq 0$ such that
\begin{equation}
M-\alpha
H \succ
0. 
\end{equation}
\end{theorem}
Note that in \cite[Thm. 4.10]{c21} it is in fact required that $H_{11}-H_{12}H_{22}^\dagger H_{12}^\top \succeq 0$ (i.e., the generalized Schur complement of $H$ with respect to $H_{22}$ is positive semidefinite). This condition is equivalent to the condition that $S_{H}$ is not an empty set, as shown in \cite[page 6]{c21}.

\section{Problem definition}
\label{Problem definition label}

\subsection{The plant and its state-space realization}
Consider the discrete-time LTI system 
\begin{equation}
\label{ARM}
\mathcal{A}(q^{-1})y(t)=\mathcal{B}_u(q^{-1})u(t)+\mathcal{B}_w(q^{-1}) w(t)+B_{d0} d(t),
\end{equation}
where $y (t) \in \mathbb{R}^p$ is the output which is measurable
for control purpose, $u(t) \in \mathbb{R}^m$ is the control input, $w(t) \in \mathbb{R}^{m_w}$ is the disturbance input associated with the performance channel,  $d(t) \in \mathbb{R}^{m_d}$ is a disturbance input (noise) which is not considered as a part of the performance channel, and $q^{-1}$ represents the delay operator, i.e., $q^{-1}y(t)=y(t-1)$.

Alternatively, the system  \eqref{ARM} can be presented in a standard state-space form. 
For future reference, let $n$ denote the order of a minimal state-space realization of the system \eqref{ARM}.
We assume that the matrix $B_{d0}\in \mathbb{R}^{p\times m_d}$ is a full column rank matrix and the matrices $\mathcal{A}(\xi)\in\mathbb{R}^{p \times p} [\xi]$, $\mathcal{B}_u(\xi)\in\mathbb{R}^{p \times m} [\xi]$ and $\mathcal{B}_w(\xi)\in\mathbb{R}^{p \times {m_w}} [\xi]$ are
matrices of polynomials in the indeterminate $\xi$, that is
\begin{equation}
\begin{split}
\mathcal{A}(\xi)&:=I+A_1\xi+A_2\xi^2+\cdots+A_l\xi^l, \\
\mathcal{B}_u(\xi)&:=B_{u0}+B_{u1}\xi+B_{u2}\xi^2+\cdots+B_{ul}\xi^l, \\
\mathcal{B}_w(\xi)&:=B_{w0}+B_{w1}\xi+B_{w2}\xi^2+\cdots+B_{wl}\xi^l.  
\nonumber
\end{split}
\end{equation}
The system \eqref{ARM} can be represented in the following specific state-space form
\begin{subequations}
\label{eq:3:11}
\begin{align}
\chi (t+1)&=
\hat{A} \chi(t)+\hat{B}u(t)+\hat{B}_1 w(t)+\hat{B}_d d(t), \label{eq:3:11-A} \\ 
y(t) &= 
 \bar{A}
 \chi(t)+B_{u0} u(t)+B_{w0} w(t)+B_{d0} d(t), \label{eq:3:11-B}
 \end{align}
\end{subequations} 
 where 
 \begin{equation}
 \label{ChiDefinition}
\chi(t):=\text{col} (
\begin{pmatrix}
y(t-1) \\ \vdots \\ y(t-l)
\end{pmatrix}, 
\begin{pmatrix}
u(t-1) \\ \vdots \\ u(t-l)
\end{pmatrix}, 
\begin{pmatrix}
w(t-1) \\ \vdots \\ w(t-l)
\end{pmatrix}
)
\end{equation} 
is the system state, $\hat{B}_d:=\text{col}(B_{d0},0)$ and 

 \begin{equation}
 \nonumber
 \bar{A}:= \begin{pmatrix} \underline{A} & \underline{B}_u & \underline{B}_w  \end{pmatrix}, 
\end{equation}
\begin{equation}
\nonumber
 \underline{A}:=
 \begin{pmatrix}
 -A_1 & -A_2 & \cdots &  -A_l 
 \end{pmatrix}, 
 \end{equation}
 \begin{equation}
\nonumber
  \underline{B}_u :=
 \begin{pmatrix}
 B_{u1} & B_{u1} & \cdots &  B_{ul} 
 \end{pmatrix}, 
 \end{equation}
 \begin{equation}
\nonumber
  \underline{B}_w :=
 \begin{pmatrix}
 B_{w1} & B_{w1} & \cdots &  B_{wl} 
 \end{pmatrix},  
  \end{equation}
 \begin{equation}
 \label{ExtendedStateRealization}
\hat{A} := \begin{pmatrix} \bar{A} \\ J_{A} \end{pmatrix},
\quad 
 \hat{B} :=\begin{pmatrix}
B_{u0} \\  
J_{B}
 \end{pmatrix}, 
 \quad
  \hat{B}_1
 :=
\begin{pmatrix}
B_{w0} \\  
J_{B_1}
 \end{pmatrix},
 \end{equation}
 \begin{equation}
\nonumber
\small{J_A}  \small{:=
\begin{pmatrix}
\begin{pmatrix} I_{p(l-1)} & 0 \end{pmatrix} & 0  & 0\\
0_{m\times pl} & 0 & 0 \\
 0 & \begin{pmatrix} I_{m(l-1)} & 0 \end{pmatrix} & 0 \\
 0 & 0_{m_w\times ml} & 0_{m_w \times m_w l} \\
  0 & 0 & \begin{pmatrix} I_{m_w(l-1)} & 0 \end{pmatrix} \\
 \end{pmatrix}}, 
 \end{equation}
 \begin{equation}
\nonumber
J_B :=\begin{pmatrix}
0_{(pl-p) \times m} \\ I_m \\ 0_{(ml-m+m_wl) \times m} 
\end{pmatrix}, 
\
 J_{B_1}:=\begin{pmatrix}
0_{(pl-p+ml)\times m_w} \\ I_{m_w} \\ 0_{(m_wl-m_w) \times m_w}
\end{pmatrix}.
 \end{equation}

In connection to the system \eqref{ARM}, we further define the controlled output $z(t) \in \mathbb{R}^{p_z}$ as follows 
\begin{equation}
\label{eq:3:11-C}
 z(t)=\hat{C}_1\chi (t) +E u(t) +D_1 w(t), 
 \end{equation}
where $\hat{C}_1 \in \mathbb{R}^{p_z\times (p+m+m_w)l}$, $E \in \mathbb{R}^{p_z\times m}$ and $D_1 \in \mathbb{R}^{p_z\times m_w}$. The complete open-loop system is therefore given by \eqref{eq:3:11}, \eqref{eq:3:11-C}, where  $w \rightarrow z$ is the \emph{performance channel} on which we impose the desired closed loop specifications in terms of dissipativity or $H_2$ performance.  
The channel $u \rightarrow y$ is available for control.

\begin{remark} 
\label{RemarkAR}
The system model \eqref{ARM} belongs to the class of AR models, where $l$ denotes the system lag (observability index of the minimal realization of the system).

We treat the output $y$ as a signal that is measured in real time and is available for control. However, the results of the paper can be directly applied in the case that only some elements of the output vector $y$ are used for control, while the complete $y$ is used in controller synthesis only. More precisely, the following scenario is possible: during the data recording process we measure (collect data) of all elements of $y$, while the controller is designed in which the controller input is formed of only a subset of elements of $y$. The remaining part of $y$ (which is not controller input) enhances the controller synthesis as the synthesis exploits the information of the system's dynamics that is present in these signals/elements as well.

The disturbance $d$ is not part of the performance channel $w \rightarrow z$ and represents noisy disturbance (noise present during data recording process) or could be used to model  uncertainties, with control robustification as the goal.

Finally, the corresponding state space realization given by \eqref{eq:3:11} is characterized by highly structured state-space matrices \eqref{ExtendedStateRealization}. Exploiting this structure plays an important part in devising data-driven controller synthesis solutions.~$\hfill{\Box}$  
\end{remark}

\subsection{The unknown matrices and the known data}
We assume that we have no knowledge of the matrices $\bar{A}$, $B_{u0}$ and $B_{w0}$, and therefore also $\hat{A}$, $\hat{B}$, $\hat{B}_1$, see \eqref{ExtendedStateRealization}. However, we assume that we know a finite time trajectory (of length $l+N $ time steps) of the control input $u$, the disturbance input $w$ and the corresponding output trajectory $y$. The  disturbance input $d$ is not known. 
More precisely, the input-output data represented by matrices 
\begin{equation}
\label{RecordedData}
\begin{split}
Y & :=
\begin{pmatrix}
y(0) & y(1) & \cdots & y(N-1)
\end{pmatrix}), \\ 
X & :=
\begin{pmatrix}
\chi(0) & \chi(1) & \cdots & \chi(N-1)
\end{pmatrix}, \\
U & :=
\begin{pmatrix}
u(0) & u(1) & \cdots & u(N-1)
\end{pmatrix}, \\
W & :=
\begin{pmatrix}
w(0) & w(1) & \cdots & w(N-1)
\end{pmatrix}
\end{split}
\end{equation}  
is known, while the disturbance data matrix 
\begin{equation}
\label{NoiseData}
W_d :=
\begin{pmatrix}
d(0) & d(1) & \cdots & d(N-1)
\end{pmatrix}
\end{equation}
is unknown. 
The data matrices satisfy the 
equation
\begin{equation}
\label{DataModel}
Y = 
 \bar{A}
X+B_{u0} U+B_{w0} W+B_{d0}W_d.
\end{equation}

\begin{remark}
Trajectory of $w$ is considered to be measured during the data recording process.  Defined as the disturbance in the performance channel, one might ask whether it is then reasonable to know this signal during the conducted recording experiment.
We argue that for many real systems this is indeed the case.
For example, $w$ might not in fact be disturbance but reference signals, while the output of the performance channel contains the tracking errors (this case is presented as an example in Section~\ref{Example1}). Reference signals are known during the data recording process.  

Alternatively, $w$ might indeed be a disturbance, but it can be enforced in a controlled way during the experiment. For example, in a mechanical system we can often enforce forces in a controlled way, which will during exploitation in the system be uncontrolled disturbances. As example of such case in presented in Section~\ref{Example2}. There, active suspension of a car is considered and $w$ is a road profile. Indeed, this profile can be designed/measured during the data recording process.

Furthermore, we argue that such problem setting (in which $w$ is recorded during the experiment), is necessary if we lack the knowledge of how the disturbance enters the dynamics. Then the data has to replace the model, and for information to be present in the data, the signal $w$ must be known and persistently exciting of a sufficiently large order.

\end{remark}

\subsection{Assumptions regarding the known data}
Data-driven controller synthesis methods necessarily rely some assumptions regarding the data (e.g., persistency of excitation). In particular, all the results mentioned in Section~\ref{IntroductionLabel} rely on some constructed system realization 
for which the state data matrix has full row rank.

Along the same line, 
we introduce the following assumption.
\begin{assumption}
\label{Assumption-1}
$\chi(t)\in \text{im}(X)$ for all $t\in \mathbb{N}_{\geq 0}$ and any $u(t) \in \mathbb{R}^m$, $w(t)\in \mathbb{R}^{m_w}$ and $d(t) \in \mathbb{R}^{m_d}$.~$\hfill{\Box}$
\end{assumption}
The above assumption is verifiable under the following conditions. Consider the following compact\footnote{The singular value matrix is a square matrix which contains only the non-zero singular values.} singular value decomposition 
\begin{equation}
\label{SVD}
X=X_{s}\Sigma_{\chi} X_{r}^\top, 
\end{equation}
and introduce the following abbreviation $X_d:=\Sigma_{\chi} X_{r}^\top$.
The matrices $X_{s} \in \mathbb{R}^{(p+m+m_w)l\times \tilde{n}}$ and $X_d \in \mathbb{R}^{\tilde{n}\times N} $ are a full column rank matrix and a full row rank matrix, respectively.
Furthermore, note that from \eqref{eq:3:11-A} it follows that 
 Assumption~\ref{Assumption-1} holds if and only if $\text{im}(\begin{pmatrix} \hat{A} X_{s} & \hat{B} &  \hat{B}_1 & \hat{B}_d \end{pmatrix} )\subseteq \text{im}(X_{s})$.
Therefore, we can define conditions under which Assumption~\ref{Assumption-1} holds using the persistency of excitation 
\cite[Corollary 2 (iii)]{c1} and  output representation \cite[Expression (4)]{R33}:
\begin{itemize}
\item[\emph{i})]
In the noiseless case, if the input $\text{col}(u,w)$ is persistently exciting of order $n+l$, then Assumption~\ref{Assumption-1} holds if and only if $\text{im}(\hat{B}_d) \subseteq  \text{im}(X)$.
\item[\emph{ii})] In the noisy case, if the input $\text{col}(u,w,d)$ is persistently exciting  of order $n+l$, then Assumption~\ref{Assumption-1} holds. 
Although we do not control the input $d$,  the Assumption~\ref{Assumption-1} is achieved for almost any\footnote{The set of signals $d$ which would ``cancel out'' effects of $u$ and $w$ on the system in a way that results in reduction of $\text{im}(X)$ is of measure zero.}  $d$ if the input $\text{col}(u,w)$ is persistently exciting of order $n + l$, and the assumption can be verified by checking whether $\text{im}(\hat{B}_d) \subseteq \text{im}(X)$.

\end{itemize}
Note that the matrix $\hat{B}_d$, which appears in both \emph{i}) and \emph{ii}),  is assumed to be known. 
Thus, both \emph{i}) and \emph{ii}) can be verified after the data recording process. Additionally, the stated conditions can be used to design the data recording experiment (requirements of persistency of excitation). They also provide information of when the data recording can be terminated (when $\text{im}(\hat{B}_d) \subseteq \text{im}(X)$).

Next, we make an assumption which will be instrumental for the construction of controller synthesis methods.
\begin{assumption}
\label{Assumption-2}
First $p$ rows of matrix $X$ are linearly independent.~$\hfill{\Box}$
\end{assumption}

Recall that $p$ denotes the number of elements in $y(t)$, i.e., $y(t)\in \mathbb{R}^p$. 
The statement in Assumption~\ref{Assumption-2} is necessarily satisfied when the output $y$ is controllable with respect to some input which is persistently exciting of order $n+1$ \cite[Corollary 2 (ii)]{c1}. This follows directly from the formulation of the output $y$ based on matrices of some minimal realization of the system \eqref{ARM} with accordance to persistently exciting input.

\subsection{The disturbance model}
\label{The disturbance modelSubsection}
The disturbance $d$ affecting the data recording is assumed to be bounded. This is modeled using a quadratic matrix inequality (QMI) imposed on $W_d$, as done in \cite{c6}. 
\begin{assumption}
\label{Assumption1}
The matrix $W_d$ from \eqref{NoiseData} satisfies the inequality  
\begin{equation}
\label{eq:2:6}
\begin{pmatrix}
I \\
W_d^\top\\
\end{pmatrix}^\top
\underbrace{
\begin{pmatrix}
\Phi_{11} & \Phi_{12}\\
\Phi_{12}^\top & \Phi_{22}
\end{pmatrix}}_\Phi
\begin{pmatrix}
I \\
W_d^\top\\
\end{pmatrix} \succeq 0,
\end{equation}
where $\Phi=\Phi^\top \in \mathbb{R}^{(m_d+N)\times(m_d+N)}$ is a known matrix with $\Phi_{22} \prec 0$. ~$\hfill{\Box}$   
\end{assumption}
Note that $ \Phi_{22} \prec 0$, $\Phi_{22} \in \mathbb{R}^{N \times N}$, ensures that the set of matrices $W_d$ which satisfy \eqref{eq:2:6} is bounded.

\subsection{The set of plants consistent with data}
\label{SetOfSystemsSubsection}

Recall that the matrix $ \bar{A}$ is an unknown matrix, thus, equation \eqref{DataModel} can be rewritten in the following manner
\begin{equation}
\label{DataModel-1}
Y = \bar{A}_s
X_d+B_{u0} U+B_{w0} W+B_{d0}W_d,
\end{equation}
where
\begin{equation}
\nonumber
\bar{A}_s:=\bar{A}X_{s}.
\end{equation}
Note that since $X_{s}$ has a full column rank, the above definition of $\bar{A}_s$ implies that $\bar{A}_s$ can be arbitrary, i.e., the definition itself imposes no constraints on it (e.g., constraints as predefined rank).
Let $\Sigma$ be the set of all triples ($ \bar{A}_s $, $B_{u0}$, $B_{w0}$) which can explain the data ($Y$, $X_d$, $U$, $W$, $W_d$), that is
\begin{equation}
\begin{split}
\Sigma :=
\{(\bar{A}_s,B_{u0}, B_{w0}) \, &| \, \eqref{DataModel-1} \text{ holds for some }  W_d
\\ &
 \text{ which satisfy }  \eqref{eq:2:6}\}. 
\end{split}
\nonumber
\end{equation}
The set $\Sigma$ can be described using a QMI, as presented in the following lemma.  
\begin{lemma}
\label{Lemma3}
The set $\Sigma$ is a set of all triples ($\bar{A}_s $,$B_{u0}$,$B_{w0}$) that satisfy the following inequality
\begin{equation}
\label{eq:2:3}
\begin{pmatrix}
I   \\ 
\bar{A}_s^\top \\
B_{u0}^\top  \\
B_{w0}^\top  \\
\end{pmatrix}^\top
\underbrace{
\begin{pmatrix}
H_{11}  & H_{12}  & H_{13}  & H_{14}  \\
H_{12}^\top & H_{22}  & H_{23} & H_{24} \\
H_{13}^\top & H_{23}^\top  & H_{33} & H_{34} \\
H_{14}^\top & H_{24}^\top  & H_{34}^\top & H_{34} \\
 \end{pmatrix}}_{H}
\begin{pmatrix}
I   \\ 
\bar{A}_s^\top \\
B_{u0}^\top  \\
B_{w0}^\top  \\
\end{pmatrix} \succeq 0,
\end{equation}
where
\begin{equation}
\nonumber
H:=
\begin{pmatrix}
\star
\end{pmatrix}
\begin{pmatrix}
\Phi_{11}  & \Phi_{12}\\ 
\Phi_{12}^\top & \Phi_{22}
\end{pmatrix}
\begin{pmatrix}
B_{d0} & Y  \\ \hline
0 &  -X_d \\
0 & -U \\
0 & -W
\end{pmatrix}^\top.
\end{equation}
\end{lemma}
\begin{proof}
If we multiply \eqref{eq:2:6} with matrices $B_{d0}$ and $B_{d0}^\top$ from the left and right side, respectively, then we obtain an equivalent matrix inequality since matrix $B_{d0}$ has a full column rank. Next, if we substitute  $B_{d0} W_d$  using \eqref{DataModel-1} we obtain \eqref{eq:2:3}.
\end{proof}

Finally, we will need the following assumption, which will be instrumental for application of the matrix S-lemma in the proof of the main result.
 \begin{assumption}
\label{Assumption-3}
The set $\Sigma$ is a nonempty set and we have that $\text{rank}(\text{col}(X_d,U,W))=\tilde{n}+m+m_w$.
~$\hfill{\Box}$
\end{assumption}

Recall that $\text{rank}(X_s)=\tilde{n}$, see \eqref{SVD}.
Thus, the rank condition from  Assumption~\ref{Assumption-3} can be achieved/verified in the same manner as done in connection to Assumption~\ref{Assumption-1}, with the difference that we need $n+l+1$ as the order of persistency of excitation instead of $n+l$.
Note with  Assumption~\ref{Assumption-3} we have 
\begin{equation}
\label{SigmaIsBounded}
\begin{pmatrix}
 H_{22}  & H_{23} & H_{24} \\
 H_{23}^\top  & H_{33} & H_{34} \\
H_{24}^\top  & H_{34}^\top & H_{34} \\
 \end{pmatrix} \prec 0,
\end{equation}
thus, set $\Sigma$ is bounded \cite{c21}.

\subsection{The data-driven realization of the closed-loop system}
\label{SynthesisConstructionTools}

Let $X_{s\perp}$  be a full column rank matrix whose columns form a basis of $\text{ker}(X_{s}^\top)$ and let $\chi_{s}(t)$ and $\chi_{s\perp}(t)$ be uniquely defined with the following relation
\begin{equation}
\nonumber
\chi(t)=
\begin{pmatrix}
X_{s} & X_{s\perp}
\end{pmatrix} 
\begin{pmatrix}
\chi_{s}(t) \\
\chi_{s\perp}(t)
\end{pmatrix}.
\end{equation}
With Assumption~\ref{Assumption-1}, the open-loop system $\text{col}(w,u)\rightarrow z$ is described  using the following realization
\begin{equation}
\label{ReduciraniModel}
\begin{split}
\begin{pmatrix}
\chi_s(t+1) \\ z(t) 
\end{pmatrix}=
\begin{pmatrix}
A & B_1 & B \\
C_1 & D_1 & E
\end{pmatrix}
\begin{pmatrix}
\chi_s (t) \\ w(t) \\ u(t)
\end{pmatrix},
 \end{split}
\end{equation}
where $A:=X_{s}^\top \hat{A} X_{s}$, $B_1:=X_{s}^\top \hat{B}_1$, $B:=X_{s}^\top \hat{B}$ and $C_1:=\hat{C}_1X_{s}$. Note that $A \in \mathbb{R}^{\tilde{n}\times \tilde{n}}$, where $\tilde{n}$ is the dimension of the $\text{im}(X)$, see \eqref{SVD}, 
and that the system realization \eqref{ReduciraniModel} is valid due to the invariance of the state subspace based on persistency of excitation (Assumption~\ref{Assumption-1}).

In addition to the open-loop system \eqref{ReduciraniModel} we define the following output available for control 
\begin{equation}
\label{eq:3:11-D}
y_c(t)= \hat{C}  \chi(t)=C \chi_s(t),
\end{equation}
where $C:=\hat{C} X_s$. Note that we construct the matrix $\hat{C}$ to chose arbitrary elements 
 of the state vector \eqref{ChiDefinition} which are going to be used as controller input (except elements of the disturbance input $w$). See also Remark~\ref{RemarkAR}.

Next, consider the following dynamic output-feedback controller
\begin{equation}
\label{controller}
\begin{pmatrix}
x_{c} (t+1)\\
u (t)
\end{pmatrix}
=
\begin{pmatrix}
A_{c} & B_c \\
C_c & D_c
\end{pmatrix}
\begin{pmatrix}
x_{c} (t) \\
y_c (t)
\end{pmatrix},
\end{equation}
where the controller is of the same order as the open-loop system defined in \eqref{ReduciraniModel}, i.e.,  $x_{c} (t) \in \mathbb{R}^{\tilde{n}}$. 
The system \eqref{ReduciraniModel} and the controller \eqref{controller} 
form the closed-loop system
\begin{equation}
\label{closed_loop}
\begin{pmatrix}
\xi (t+1) \\
z (t)
\end{pmatrix}
=
\begin{pmatrix}
\tilde{A} & \tilde{B} \\
\tilde{C} & \tilde{D} 
\end{pmatrix}
\begin{pmatrix}
\xi (t) \\ w (t)
\end{pmatrix},
\end{equation}
where $\xi (t)=\text{col}(\chi_s (t),x_{c} (t))$  and
\begin{equation}
\label{closed_loop_matrices}
\begin{pmatrix}
\tilde{A} & \tilde{B} \\
\tilde{C} & \tilde{D} 
\end{pmatrix}
=
\begin{pmatrix}
\begin{array}{cc|c}
A+BD_cC & BC_c & B_1 \\
B_cC & A_c & 0 \\ \hline
C_1+ED_cC & EC_c & D_1
\end{array}
\end{pmatrix}.
\end{equation}

\subsection{The control problem}
\label{TheControlProblem}

\label{TheControlProblem}
In this paper, we are concerned with the following control problem. 
\begin{problem}
\label{Problem1}
Consider the system \eqref{ARM}, the controlled output \eqref{eq:3:11-C}, the output available for control \eqref{eq:3:11-D}, the controller \eqref{controller}, and suppose that we have the following knowledge: the recorded input-output data ($Y$, $X$, $U$, $W$) as defined in \eqref{RecordedData}; the disturbance bound  specified by $\Phi$ in Assumption~\ref{Assumption1}.
Let Assumptions~\ref{Assumption-1}, \ref{Assumption-2} and \ref{Assumption-3} hold.  
Recall that the considered system has the representation \eqref{ReduciraniModel}, and with controller \eqref{controller} it forms the closed-loop system \eqref{closed_loop}.
Consider two separate synthesis goals:
\begin{enumerate}
 \item[\emph{a})]
For a given performance supply function  \eqref{supplyFunction} with  $R \succeq 0$,
design an  output-feedback controller \eqref{controller} such that the closed-loop system \eqref{closed_loop} is stable and strictly dissipative with respect to a given supply function for all triples $(\bar{A}_s, B_{u0}, B_{w0}) \in \Sigma$.
 \item[\emph{b})] 
For a given $H_2$ performance $\mu$, 
design a output-feedback controller \eqref{controller} such that the performance channel $w \rightarrow z$ of the closed-loop system \eqref{closed_loop} achieves a given $H_2$ performance for all  triples $(\bar{A}_s, B_{u0}, B_{w0}) \in \Sigma$.~$\hfill{\Box}$
\end{enumerate}   
\end{problem}

Recall that the disturbance $w$ is used only for the purpose of
controller synthesis, while the designed controller uses only output $y$ and input $u$, i.e., the full state is not used for control in real time. 
    
\section{Controller synthesis}
\label{SolutionProblem1}

The two theorems presented in this section provide the constructive controller synthesis methods for the control problems defined in the previous section. Note that the statements in these theorems are indeed the statements of the dissipativity and the $H_2$ performance level, as presented in Sections~\ref{SecDissipativity} and \ref{H2performanceSubsection}, respectively.
 
\begin{figure*}[b]
 \par\noindent\rule{\textwidth}{0.5pt}
\begin{equation}
\tag{CS}
\label{eq:3:5}
\begin{gathered}
\Pi:=
\begin{pmatrix}
\Pi_{11} & \Pi_{12} \\
\Pi_{12}^\top & \Pi_{22} 
\end{pmatrix}, \quad
\Pi_{11}:=
\begin{pmatrix}
\bar{X} & I & \star  & \star   \\
I & \bar{Y} & \star   & \star\\
-S(C_1+ED_cC) &  -S(C_1\bar{Y}+E\tilde{C}_c) & -Q-\text{He}\{SD_1 \}  &   \star \\
T^\top(C_1+ED_cC) & T^\top(C_1\bar{Y}+E\tilde{C}_c) & T^\top D_1 & \tilde{R}^{-1}
\end{pmatrix}, \\
\Pi_{12}^\top:= 
\begin{pmatrix}
\check{A}+\check{B}D_cC & \check{A}\bar{Y}+\check{B}\tilde{C}_c  & \check{B}_1 & 0_{\tilde{n} \times \tilde{p}_z} \\
\bar{X}\check{A}+\tilde{B}_cC& \tilde{A}_c   &  \bar{X}\check{B}_1  & 0 \\
I & \bar{Y} & 0 & 0 \\
D_cC & \tilde{C}_c & 0 & 0 \\
0 & 0 & I_{m_w} & 0 \\
0_{\tilde{p} \times \tilde{n}} & 0 & 0 & 0
\end{pmatrix}, 
\\
\Pi_{22}:=
\begin{pmatrix}
\bar{Y} & I & \star  & \star  &  \star &  \star \\
I & \bar{X} &  \star  & \star &  \star &  \star \\
\begin{pmatrix} 0 &  -\alpha \bar{H}_{12}^\top  \end{pmatrix} L^{-\top}  & \begin{pmatrix} 0 &  -\alpha \bar{H}_{12}^\top  \end{pmatrix}  L^{-\top}\bar{X} &   -\alpha  H_{22} & \star & \star &  \star \\
\begin{pmatrix} 0 &  -\alpha \bar{H}_{13}^\top  \end{pmatrix} L^{-\top} & \begin{pmatrix} 0 &  -\alpha \bar{H}_{13}^\top  \end{pmatrix}  L^{-\top}\bar{X}   &   -\alpha  H_{23}^\top & -\alpha  H_{33}^\top & \star  &  \star \\
\begin{pmatrix} 0 &  -\alpha \bar{H}_{14}^\top  \end{pmatrix} L^{-\top}  & \begin{pmatrix} 0 &  -\alpha \bar{H}_{14}^\top  \end{pmatrix}  L^{-\top}\bar{X}   &   -\alpha  H_{24}^\top & -\alpha  H_{34}^\top & -\alpha H_{44} &  \star\\
\begin{pmatrix} 0 & \alpha \Xi^\top \end{pmatrix} L^{-\top} & \begin{pmatrix} 0 & \alpha \Xi^\top \end{pmatrix} L^{-\top}\bar{X}  &  0 & 0 & 0 &  \alpha \Lambda^{-1}  \\
\end{pmatrix}, 
\\
\Delta :=
\begin{pmatrix}
I & D_1^\top & 0 & 0 \\
\star & Z & C_1+ED_cC & C_1 \bar{Y}+E \tilde{C}_c \\
\star & \star & \bar{X} & I \\
\star & \star & \star & \bar{Y}
\end{pmatrix}.
\end{gathered}
\end{equation}
\end{figure*}
\begin{theorem}
\label{Theorem1}
Consider the Problem~\ref{Problem1}a, the set of matrices defined in Section~\ref{Appendix0} in the Appendix and the matrix  $\Pi$ defined in expression \eqref{eq:3:5} presented at the bottom of next page. In connection to the matrix $R\succeq 0$, from the definition of the supply rate, let  $\tilde{R} \in \mathbb{R}^{\tilde{p}_z \times \tilde{p}_z  }$ be a positive definite matrix defined by arbitrary factorization  $R=T\tilde{R} T^\top$.
Then, the following two statements are equivalent:
\begin{enumerate}
 \item[\emph{i})]
The inequalities
 \begin{subequations}
 \label{eq:3:4}
 \begin{align}
 & P \succ 0, \label{E1} \\ 
 & \begin{pmatrix}
\star
\end{pmatrix}^\top
\begin{pmatrix}
-P & 0 & 0 & 0  \\
0& P & 0 & 0   \\
0 & 0 & Q & S\\
0 & 0 & S^\top & R\\
\end{pmatrix}\begin{pmatrix}
I & 0 \\
\tilde{A} & \tilde{B} \\
0 & I\\
\tilde{C} & \tilde{D}  \\
\end{pmatrix}
\prec 0, \label{E2}
 \end{align}
 \end{subequations}
hold for all $(\bar{A}_s, B_{u0}, B_{w0})\in \Sigma$.
\item[\emph{ii})]
The inequality
\begin{equation}
\label{eq:600:301}
\Pi \succ 0,
\end{equation}
holds for some real scalar $\alpha$, symmetric matrices $\bar{X}$ and $\bar{Y}$, unstructured matrices $\tilde{A}_c$, $\tilde{B}_c$, $\tilde{C}_c$ and $D_c$. 
\end{enumerate}
Furthermore, the controller parameters can be reconstructed from \eqref{eq:600:301} as follows
\begin{equation}
\label{eq:3:18}
\begin{gathered}
V^\top=\bar{U}^{-1}(I-\bar{X}\bar{Y}), \\
B_c=\bar{U}^{-1}(\tilde{B}_c-\bar{X}\check{B}D_c), \\
C_c=(\tilde{C}_c-D_cC\bar{Y})V^{-\top}, \\
A_c=\bar{U}^{-1}(\tilde{A}_c-\bar{X}(\check{A}+\check{B} \tilde{C}_c)-\bar{U}B_cC\bar{Y})V^{-\top},
\end{gathered}
\end{equation}
where $U$ is an arbitrary full rank matrix. 
\end{theorem}
\begin{proof}
The proof is divided in several steps. 
In step 1, we rewrite \eqref{eq:3:4} in a form suitable for applying Theorem~\ref{Theorem0}. In step 2, we apply Theorem~\ref{Theorem0}. In step 3, we transform derived conditions to obtain \eqref{eq:600:301} and \eqref{eq:3:18}, thus, prove necessity. In step 4, we provide the final argumentation for sufficiency.

\noindent \emph{Step 1.}
Suppose that  \eqref{eq:3:4} holds.
By applying the Schur complement rule on \eqref{E2} with respect to the matrix $\text{diag}(-\tilde{R},-P)$, we obtain the equivalent expression to \eqref{eq:3:4}:
\begin{equation}
\label{eq:3:12}
\Theta:=
\begin{pmatrix}
\Theta_{11} & \star\\  
\begin{pmatrix}
T^\top \tilde{C} & T^\top \tilde{D} \\
\tilde{A} & \tilde{B}
\end{pmatrix} & \begin{pmatrix}
\tilde{R}^{-1} & 0 \\
0 & P^{-1}
\end{pmatrix}
\end{pmatrix}
 \succ 0,
\end{equation}
where 
\begin{equation}
\nonumber
\Theta_{11}:=
\begin{pmatrix}
\star
\end{pmatrix}^\top
\begin{pmatrix}
P & 0 & 0 \\
0 & -Q & -S \\
0 & -S^\top & 0
\end{pmatrix}
\begin{pmatrix}
I & 0 \\
0 & I \\
\tilde{C} & \tilde{D}
\end{pmatrix}.
\end{equation}
Let 
\begin{equation}
\label{P_MatrixStructure}
\begin{gathered}
P:=
\begin{pmatrix}
\bar{X}  & \bar{U} \\
\bar{U}^\top & M_{1}
\end{pmatrix}, \quad 
P^{-1}:=
\begin{pmatrix}
\bar{Y} & V \\
V^\top & M_{2}
\end{pmatrix}
, \nonumber 
\end{gathered}
\end{equation}
where the matrices $\bar{X}$, $M_1$, $\bar{Y}$ and $M_2$ are of size $\tilde{n}\times \tilde{n}$.
Note that according to the multiple controller realization argumentation shown in 
the Appendix, Section~\ref{AppendixA}, we can choose the matrix $\bar{U}$. For the purpose of constructing controller synthesis we
assume that $\bar{U}$ is an arbitrary regular matrix.
Next, we apply a congruence transformation on \eqref{eq:3:12} with respect to the matrix
$S_1:=\text{diag }(I, S_2 )$ in order to obtain the equivalent inequality 
\begin{equation}
\tilde{\Theta}:=S_1^\top\Theta S_1\succ 0, \nonumber 
\end{equation}
where $S_2:=\begin{small} \begin{pmatrix} 0 &  L^\top \\  I_{\tilde{n}} & 0 \end{pmatrix} \end{small}$. Recall \eqref{closed_loop_matrices}, see definitions of the matrices $L$ and $K$ and the relation \eqref{ipakTreba1} presented in Appendix~\ref{Appendix0}. We have the following equalities
\begin{equation}
\nonumber
\begin{split}
L(A+BD_cC) = &\begin{pmatrix} 0 \\ \bar{A}_s
\end{pmatrix}
+\begin{pmatrix} 0 \\ B_{u0} \end{pmatrix}D_cC +
\\ &
\begin{pmatrix} K_{12} \\ K_{22} \end{pmatrix}( J_A X_s+J_BD_cC),
\end{split}
\end{equation}
\begin{equation}
\nonumber
LBC_c=\begin{pmatrix} 0 \\ B_{u0}  \end{pmatrix}C_c + \begin{pmatrix} K_{12} \\ K_{22} \end{pmatrix} J_B C_c,
\end{equation}
\begin{equation}
\nonumber
L B_1=\begin{pmatrix} 0 \\ B_{w0} \end{pmatrix} + \begin{pmatrix} K_{12} \\ K_{22} \end{pmatrix}  J_{B_1},
\end{equation}
\begin{equation}
\nonumber
S_2^\top
P^{-1} 
S_2=
\begin{pmatrix}
M_2 & V^\top L^\top \\
L V & L \bar{Y} L^\top
\end{pmatrix}.
\end{equation}
Let $\tilde{\Theta}_{11}$ be the upper left block diagonal matrix of $\tilde{\Theta}$ with size $(4\tilde{n}+m_w+\tilde{p}_z-p) \times  (4\tilde{n}+m_w+\tilde{p}_z-p)$, and  $\tilde{\Theta}_{22}$ be the lower right block diagonal matrix with size $p \times p$.
Then, by applying the Schur complement rule on $\tilde{\Theta} \succ 0$ with respect to the matrix $\tilde{\Theta}_{11}$, we obtain the equivalent condition
\begin{equation}
\label{eq:300:13}
\Gamma
(\begin{pmatrix}  \tilde{\Theta}_{22} & 0 \\ 0 & 0_{\tilde{n}+m+m_w} \end{pmatrix}-
\tilde{\Gamma} \tilde{\Theta}_{11}^{-1} 
\tilde{\Gamma}^\top
)
\Gamma^\top 
\succ 0, \quad \tilde{\Theta}_{11}\succ 0,
\end{equation}
where $\Gamma:=\begin{pmatrix} I & \bar{A}_s & B_{u0} & B_{w0} \end{pmatrix}$ and $\tilde{\Gamma}:=\begin{pmatrix}  \tilde{\Gamma}_1 & \tilde{\Gamma}_2 \end{pmatrix}$ with
\begin{equation}
\nonumber
\begin{split}
\tilde{\Gamma}_1 &:=
\begin{pmatrix}
K_{22}(J_AX_s+J_BD_c C) & K_{22} J_B C_c & K_{22} J_{B_1} \\
I_{\tilde{n}} & 0 & 0   \\
D_cC & C_c & 0  \\ 
0 & 0 & I_{m_w}  
\end{pmatrix},
\\
\tilde{\Gamma}_2 &:=
\begin{pmatrix}
 0_{p\times \tilde{p}_z} &  \begin{pmatrix} 0 & I_p \end{pmatrix} \begin{pmatrix} LV & L\bar{Y}L^\top \end{pmatrix} 
 \begin{pmatrix} I_{2\tilde{n}-p} \\ 0 \end{pmatrix} \\
 0 & 0  \\
\end{pmatrix}.
\end{split}
\end{equation}

\noindent \emph{Step 2.}
Note that according to Lemma~\ref{Lemma3}, inequality \eqref{eq:2:3} defines the set of all triples $(\bar{A}_s,B_{u0}, B_{w0}) $.  Furthermore,
the first matrix inequality in \eqref{eq:300:13}
and QMI \eqref{eq:2:3}
 are in a form suitable for the application of Theorem~\ref{Theorem0}. Also, 
 Assumptions~\ref{Assumption1} and \ref{Assumption-3} imply that $\Sigma$ is a nonempty  and bounded set.
Therefore, we can use Theorem~\ref{Theorem0} to state that \eqref{eq:300:13} holds for all triples $(\bar{A}_s,B_{u0}, B_{w0}) $ that satisfy the QMI \eqref{eq:2:3}, if and only if there exist $\alpha \geq 0$  such that
\begin{equation}
\label{eq:3:15}
\begin{pmatrix} \tilde{\Theta}_{22} & 0 \\ 0 & 0_{\tilde{n}+m+m_w} \end{pmatrix}-\alpha H -
\tilde{\Gamma} \tilde{\Theta}_{11}^{-1} 
\tilde{\Gamma}^\top 
 \succ 0, \quad \tilde{\Theta}_{11} \succ 0.
\end{equation}
\noindent \emph{Step 3.}
Note the definitions \eqref{ipakTreba3}, \eqref{set_element}, \eqref{ipakTreba2}, presented in Appendix~\ref{Appendix0}. By applying a congruence transformation on the first inequality in \eqref{eq:3:15} with respect to the matrix
\begin{equation}
S_3:=
\begin{pmatrix}
 I & 0 & 0  & 0 \\
 \check{A}_s^\top  & I & 0 & 0  \\
 \check{B}_{u0}^\top & 0 & I & 0  \\ 
  \check{B}_{w0}^\top & 0 & 0 & I  \\ 
\end{pmatrix}, \nonumber
\end{equation}
we obtain the equivalent expression
\begin{equation}
\label{eq:333:665}
\begin{pmatrix} \tilde{\Theta}_{22} & 0 \\ 0 & 0_{\tilde{n}+m+m_w} \end{pmatrix}-\alpha \bar{H} -
\bar{\Gamma} \tilde{\Theta}_{11}^{-1} 
\bar{\Gamma}^\top 
 \succ 0, \quad \tilde{\Theta}_{11} \succ 0,
\end{equation}
where 
\begin{equation}
\label{barH}
\bar{H}:=
\begin{pmatrix}
 \Xi \Lambda \Xi^\top & \star &   \star &   \star  \\ 
\bar{H}_{12}^\top &  H_{22} &  \star &   \star \\
\bar{H}_{13}^\top &  H_{23}^\top & H_{33} &   \star \\
\bar{H}_{14}^\top &  H_{24}^\top & H_{34}^\top &    H_{34}
\end{pmatrix},
\end{equation}
and $\bar{\Gamma}:=\begin{pmatrix} \bar{\Gamma}_1 & \bar{\Gamma}_2 & \bar{\Gamma}_3 \end{pmatrix}$ with
\begin{equation}
\nonumber
 \bar{\Gamma}_1 :=
\begin{pmatrix}
\check{A}_s +\check{B}_{u0} D_c C+K_{22}(J_AX_s+J_BD_c C) \\
I_{\tilde{n}}  \\
D_c C  \\ 
0 
\end{pmatrix} ,
\end{equation}
\begin{equation}
\nonumber
 \bar{\Gamma}_2 :=
\begin{pmatrix}
\check{B}_{u0} C_c+K_{22} J_B C_c & \check{B}_{w0}+K_{22} J_{B_1}   \\
0 & 0   \\
 C_c & 0  \\ 
 0 & I_{m_w}
\end{pmatrix} ,
\end{equation}
\begin{equation}
 \bar{\Gamma}_3 :=
\begin{pmatrix}
 0_{p\times \tilde{p}_z} &  \begin{pmatrix} 0 & I_p \end{pmatrix} \begin{pmatrix} LV & L\bar{Y} L^\top \end{pmatrix} \begin{pmatrix} I_{2\tilde{n}-p} \\ 0 \end{pmatrix} \\
 0 & 0   \\
\end{pmatrix}.  \nonumber
\end{equation}

 By applying the Schur complement rule on \eqref{eq:333:665} first with respect to the matrix $\tilde{\Theta}_{11}^{-1}$, and then to the matrix $\frac{1}{\alpha} \Lambda$ (note that $\eqref{eq:333:665}$ implies that $\alpha>0$), we obtain the equivalent expression
\begin{equation}
\label{eq300:600}
\begin{pmatrix}
\tilde{\Theta}_{11}  & \bar{\Gamma}^\top  & 0 \\
\bar{\Gamma} & \bar{H}_{\Lambda}  & \Gamma_\Lambda^\top \\
0 & \Gamma_\Lambda & \alpha \Lambda^{-1}
\end{pmatrix} \succ 0, \quad 
\end{equation}
where $\bar{H}_{\Lambda}:=\text{diag}(\tilde{\Theta}_{22},0_{\tilde{n}+m+m_w})-\alpha \underline{H}$, $\underline{H}$ is the matrix $\bar{H}$ with a zero matrix instead of the matrix $\Xi \Lambda \Xi^\top$, and 
$\Gamma_\Lambda:=\begin{pmatrix} \alpha \Xi^\top & 0 \end{pmatrix}$.

Recall that $\bar{U}$ is an arbitrary regular matrix, let 
\begin{equation}
S_4:=
\begin{pmatrix}
I_{\tilde{n}} & 0 \\
0 & L^{-\top}
\end{pmatrix}
,\ 
S_5:=
\begin{pmatrix}
0 & \bar{U}^\top \\
I & \bar{X}
\end{pmatrix}, \
S_6:=
\begin{pmatrix}
I & \bar{Y} \\
0 & V^\top
\end{pmatrix},
\end{equation}
 and note that 
\begin{equation}
\label{eq:rekonstrukcija}
P P^{-1}=
\begin{pmatrix} 
\bar{X}\bar{Y}+\bar{U}V^\top & \bar{X}V+\bar{U} M_2\\
\bar{U}^\top \bar{Y}+M_1V^\top  & \bar{U}^\top V+M_1M_2
\end{pmatrix}=
\begin{pmatrix}
I & 0 \\
0 & I
\end{pmatrix}.
\end{equation}
Thus, if we apply a congruence transformation on \eqref{eq300:600} with respect to matrix
\begin{equation}
\nonumber
S_7:=
\text{diag}( 
I_{2\tilde{n}+m_w+\tilde{p}_z}, 
S_4
S_5,
I_{\tilde{n}+m+m_w+\tilde{p}}
),
\end{equation}
it follows that
\begin{equation}
\label{Statement111}
(S_4
S_5)^\top
\begin{pmatrix}
M_2 & V^\top L^{\top} \\
LV & L\bar{Y}L^{\top}
\end{pmatrix} 
S_4
S_5
=
\begin{pmatrix}
\bar{Y} & I \\
I & \bar{X}
\end{pmatrix} \succ 0.
\end{equation}
Next, if we use the Schur complement rule on \eqref{Statement111} with respect to the matrix $\bar{X}$, we obtain the equivalent statement $\bar{Y}-\bar{X}^{-1}\succ 0$ and $\bar{X}\succ 0$. From there it follows that $I-\bar{X}\bar{Y}$ is a non-singular matrix.  Since \eqref{eq:rekonstrukcija} implies that $V^\top =\bar{U}^{-1}(I-\bar{X}\bar{Y})$, we can conclude that the equation \eqref{eq300:600} implies that the matrix $V$ has full rank.
Therefore, we can apply a congruence transformation on \eqref{eq300:600} with respect to matrix
\begin{equation}
\nonumber
S_8:=
\text{diag}( 
S_6
, I_{m_w+\tilde{p}_z},
S_4 
S_5,
I_{\tilde{n}+m+m_w+\tilde{p}}
),
\end{equation}
to obtain the equivalent inequality \eqref{eq:600:301}, this proves necessity. In order to obtain \eqref{eq:600:301} and \eqref{eq:3:18}, we used \eqref{eq:rekonstrukcija}, \eqref{Statement111}, \eqref{ipakTreba1}, matrices $\check{A}$, $\check{B}$, $\check{B}_1$  defined in Appendix~\ref{Appendix0} and the following set of equalities 
\begin{equation}
\nonumber
S_6^\top P S_6= 
\begin{pmatrix}
\bar{X} & I\\
I & \bar{Y}
\end{pmatrix},
\end{equation}
\begin{equation}
\nonumber
\tilde{C}
S_6=
\begin{pmatrix}
C_1+ED_c C & C_1\bar{Y}+E\tilde{C}_c
\end{pmatrix}, 
\end{equation}
\begin{equation}
\nonumber
\tilde{C}_c:=D_c C \bar{Y}+C_cV^\top,
\end{equation}
\begin{equation}
\nonumber
\begin{pmatrix}
I & 0 \\
D_c C & C_c
\end{pmatrix}
S_6=
\begin{pmatrix}
I & \bar{Y}\\
D_c C & \tilde{C}_c
\end{pmatrix},
\end{equation}
\begin{equation}
\nonumber
\bar{\Theta}:=
\begin{pmatrix}
B_cC & A_c \\
K( \begin{pmatrix} \check{A}_s  \\ J_AX_s \end{pmatrix}   + \begin{pmatrix} \check{B}_{u0} \\ J_B \end{pmatrix} D_c C) & K\begin{pmatrix} \check{B}_{u0} \\ J_B \end{pmatrix} C_c
\end{pmatrix},
\end{equation}
\begin{equation}
\nonumber
\begin{split}
(S_4 S_5)^\top
\bar{\Theta} S_6:= 
\begin{pmatrix}
\check{A}+\check{B} D_c C & \check{A} \bar{Y}+\check{B} \tilde{C}_c \\
\bar{X}\check{A}+\tilde{B}_c C& \tilde{A}_c  
\end{pmatrix},
\end{split}
\end{equation}
\begin{equation}
\nonumber
\tilde{B}_c:=
\bar{X}\check{B} D_c+\bar{U}B_c, 
\end{equation}
\begin{equation}
\nonumber
\tilde{A}_c :=\bar{X}(\check{A}\bar{Y} +\check{B} \tilde{C}_c)+\bar{U}(B_cC\bar{Y}+A_cV^T),
\end{equation}
\begin{equation}
\nonumber
(S_4 S_5)^\top
\begin{pmatrix}
  0  \\ K \begin{pmatrix} \check{B}_{w0} \\ J_{B_1} \end{pmatrix}
 \end{pmatrix}= \begin{pmatrix}
\check{B}_1  \\ \bar{X} \check{B}_1 
\end{pmatrix},
\end{equation}
\begin{equation}
\nonumber
\hat{\Theta}:=
\begin{pmatrix}
0_{\tilde{n} \times \tilde{n}} & \begin{pmatrix} 0_{\tilde{n} \times (\tilde{n}-p)} & -\alpha \bar{H}_{12}^\top \end{pmatrix} \\
0_{m \times \tilde{n}}  & \begin{pmatrix} 0_{m \times (\tilde{n}-p)} & -\alpha \bar{H}_{13}^\top \end{pmatrix} \\
0_{m_w \times \tilde{n}}  & \begin{pmatrix} 0_{m_w \times (\tilde{n}-p)} & -\alpha \bar{H}_{14}^\top \end{pmatrix} \\
0_{\tilde{p} \times \tilde{n}}  & \begin{pmatrix} 0_{\tilde{p} \times (\tilde{n}-p)} & \alpha \Xi^\top \end{pmatrix} \\
\end{pmatrix},  
\end{equation}
\begin{equation}
\nonumber
\hat{\Theta}S_4 S_5=
\begin{pmatrix}
0_{\tilde{n} \times (\tilde{n}-p)} & - \alpha \bar{H}_{12}^\top  \\
 0_{m \times (\tilde{n}-p)} & -\alpha \bar{H}_{13}^\top   \\
0_{m_w \times (\tilde{n}-p)} & -\alpha  \bar{H}_{14}^\top   \\
 0_{\tilde{p} \times (\tilde{n}-p)} &  \alpha  \Xi^\top 
\end{pmatrix} L^{-\top} 
\begin{pmatrix}  I  & \bar{X} \end{pmatrix}.
 \end{equation}

\noindent \emph{Step 4.}
Suppose that \eqref{eq:600:301} and \eqref{eq:3:18} hold.
Therefore, we can construct matrix $V$ and controller matrices using \eqref{eq:3:18} and reverse the argumentation to prove sufficiency. This establishes the equivalence between statements \emph{i}) and \emph{ii}).
\end{proof}
\begin{theorem}
\label{Theorem2}
Consider the Problem~\ref{Problem1}b, the set of matrices defined in Section~\ref{Appendix0} in the Appendix and the matrices  $\Pi$ and $\Delta $ defined in expression \eqref{eq:3:5}. 
Let $Q=-I_{m_w}$ and $\tilde{p}_z=0$.
Then, the following two statements are equivalent:
\begin{enumerate}
 \item[\emph{i})]
The inequalities
 \begin{subequations}
 \label{eq:3:4_H2}
 \begin{align}
 \text{tr}(Z)  < \mu^2,& \label{E1_H2} \\ 
  \begin{pmatrix}
I & \tilde{B}^\top \\
\tilde{B} &  P-\tilde{A} P \tilde{A}^\top
\end{pmatrix}
  \succ 0,&  \label{E2_H2} \\
   \begin{pmatrix}
Z-\tilde{D}\tilde{D}^\top  &  \tilde{C}P \\
P  \tilde{C}^\top &  P
\end{pmatrix}  \succ 0, & \label{E3_H2} 
 \end{align}
 \end{subequations}
hold for all  $(\bar{A}_s, B_{u0}, B_{w0})\in \Sigma$.
\item[\emph{ii})]
The inequalities 
\begin{equation}
\label{eq:600:301_H2}
 \text{tr}(Z)  < \mu^2, \quad \Pi \succ 0 , \quad \Delta \succ 0,
\end{equation}
hold for some real scalar $\alpha$, symmetric matrices $X$ and $Y$, unstructured matrices $\tilde{A}_c$, $\tilde{B}_c$, $\tilde{C}_c$ and $D_c$.
\end{enumerate}
Furthermore, the controller parameters can be reconstructed from \eqref{eq:600:301_H2} as follows
\begin{equation}
\label{eq:3:18_H2}
\begin{gathered}
\bar{U}=(I-\bar{X}\bar{Y})V^{-\top}, \\
B_c=\bar{U}^{-1}(\tilde{B}_c-\bar{X}\check{B}D_c), \\
C_c=(\tilde{C}_c-D_cC\bar{Y})V^{-\top}, \\
A_c=\bar{U}^{-1}(\tilde{A}_c-\bar{X}(\check{A}\bar{Y}+\check{B}\tilde{C}_c)-\bar{U}B_cC\bar{Y})V^{-\top},
\end{gathered}
\end{equation}
where $V$ is an arbitrary full rank matrix. 
\end{theorem}
\begin{proof}
Suppose that  \eqref{eq:3:4_H2} holds.
By applying the Schur complement rule with respect to the matrix $P$ in \eqref{E2_H2} we obtain an equivalent expression
\begin{equation}
\label{h2Derivation1}
\Theta :=
\begin{pmatrix}
P^{-1} & 0 & \tilde{A}^\top\\
0 &  I & \tilde{B}^\top \\
\tilde{A} & \tilde{B} &  P
\end{pmatrix}\succ 0.
\end{equation}
Note that this matrix inequality is equal to \eqref{eq:3:12} with $Q=-I_{m_w}$, $\tilde{p}_z=0$  and 
\begin{equation}
\begin{gathered}
P:=
\begin{pmatrix}
\bar{Y} & V \\
V^\top & M_{2}
\end{pmatrix}, \quad 
P^{-1}:=
\begin{pmatrix}
\bar{X}  & \bar{U} \\
\bar{U}^\top & M_{1}
\end{pmatrix}. \nonumber \\
\end{gathered}
\end{equation}
Thus, we can follow same the same argumentation as in
the proof of Theorem~\ref{Theorem1} to obtain the matrix inequality $\Pi \succ 0$ from \eqref{eq:600:301_H2}.

Next, by applying the Schur complement rule with respect to the matrix $I_{m_w}$ in \eqref{E3_H2} we obtain an equivalent expression
\begin{equation}
\label{h2Derivation17}
   \begin{pmatrix}
   I & \tilde{D}^\top & 0 \\
\tilde{D}& Z  &  \tilde{C}P \\
0 & P  \tilde{C}^\top &  P
\end{pmatrix} \succ 0.
\end{equation}
Finally, by applying a congruence transformation on \eqref{h2Derivation17} with respect to the matrix $S_4:=\text{diag}(I,P^{-1}\begin{pmatrix} I & \bar{Y} \\ 0 & V^\top \end{pmatrix})$ we obtain the equivalent expression $\Delta \succ 0$ from \eqref{eq:600:301_H2}, what completes the proof.
\end{proof}
\begin{remark}
Matrix inequality  $\Pi \succ 0$ from 
\eqref{eq:600:301} and \eqref{eq:600:301_H2} is not  a LMI in the decision variables, because of the term $\alpha \bar{X}$. However, by fixing $\alpha \in \mathbb{R}_{> 0}$ we get an LMI that can be solved multiple times in a line search procedure (as in \cite{c4} and   \cite{R25}).
In the case LMI is not feasible, then the zero on the right hand side can be replaced with $-\varepsilon I$ and a  line search can be performed while minimizing the scalar $\varepsilon$ in the direction of decreasing its value, until $\varepsilon<0$ is eventually found, in which case inequality is feasible.
The matrix inequality $\Pi \succ 0$  differs from  LMI of the analogous model-based controller synthesis. 
Size of that LMI for model-based synthesis is $4\tilde{n} + m_w +\tilde{p}_z$ and $4\tilde{n} + m_w$ for dissipativity and $H_2$ control, respectively (see e.g. \cite[Ch. 4]{R25}). In contrast, size of our LMI is increased by $\tilde{n}+m+m_w+\tilde{p}$.
Although our synthesis method belongs to direct-data driven approach, note that we calculate parameters of one system which is consistent with the recorded data for the synthesis construction purposes. If there is no disturbance $d$ during the data recording process, we reconstruct the actual system. In that case, in the matrix $\bar{H}$, defined in \eqref{barH}, the matrices $\begin{pmatrix}
 \Xi \Lambda \Xi^\top & \bar{H}_{12} & \bar{H}_{13} & \bar{H}_{14}
\end{pmatrix}$ are replaces with zero matrices of an appropriate dimension, since $\Phi_{22}$ is only non-zero matrix in $\Phi$, see \eqref{eq:2:6}. Consequently, the diagonal block $\alpha \Lambda^{-1}$ as well as all corresponding rows and columns are omitted in the definition of $\Pi$ in \eqref{eq:3:5}. Therefore, $\Pi \succ 0$ from 
\eqref{eq:600:301} and \eqref{eq:600:301_H2} is reduced to LMI.
Furthermore,
if some pair  $(\bar{A}_s,B_{u0},B_{w0}) \in \Sigma$ is calculated using \eqref{set_element} from Appendix~\ref{Appendix0}, then $\begin{pmatrix}
 \bar{H}_{12} & \bar{H}_{13} & \bar{H}_{14}
\end{pmatrix}=0$
 also holds in the noisy case, otherwise, this may not be true.
\end{remark}

\section{EXAMPLES}
In this section we test the presented controller synthesis methods using numerical examples.
\subsection{Robust stabilization, $H_\infty$ and $H_2$ control}
\label{Example1}
We consider an unstable discrete-time system from \cite{R22}, where it was used to illustrate the controller synthesis for stabilization.
In order to introduce the performance related channel, we extend the system with a performance input-output pair.
The matrices which define the AR model are 
\begin{equation}
\nonumber
\underline{A} =
\begin{pmatrix}
0 & 1 & 0  & 0 \\
-1 & 1 & 1 & 0
\end{pmatrix},
\quad 
\underline{B}_u =
\begin{pmatrix}
2 & 0 & 0 & 0 \\
1 & 1 & 1 & -1
\end{pmatrix},
\end{equation}
\begin{equation}
\nonumber
\quad 
\underline{B}_w =
\begin{pmatrix}
2 & 0  \\
2 & 0 
\end{pmatrix}, \quad B_{u0}=\begin{pmatrix} 0 & 0 \\ 0 & 0 \end{pmatrix}, \quad B_{w0}=\begin{pmatrix} 0 \\ 0 \end{pmatrix}.
\end{equation}
 These matrices are obtained, following the procedure from \cite[Expression (4b)]{R22}, from a minimal state-space realization
\begin{equation}
\nonumber
\begin{pmatrix}
x(t+1) \\ y(t)
\end{pmatrix}
=
\begin{pmatrix}
A_m & B_m & B_{m1} \\
C_m & D_m & D_{m1}
\end{pmatrix}
\begin{pmatrix}
x(t) \\ u(t) \\ w(t)
\end{pmatrix},
\end{equation}
where the disturbance $w(t)$ is additionally introduced, along with associated matrices $B_{m1}$ and $D_{m1}$, and
\begin{equation}
\nonumber
\begin{pmatrix}
A_m & B_m & B_{m1} \\
C_m & D_m & D_{m1}
\end{pmatrix}=
\begin{pmatrix}
\begin{array}{ccc|cc|c}
0 & 1 & 0 & 1 & 0 & 1 \\
-1 & 0 & 0 & 0 & 1 & 1\\
0 & 0 & 1 & 1 & 0 & 1\\ \hline
1 & 0 & 1 & 0 & 0  & 0\\
0 & 1 & 1 & 0 & 0 & 0
\end{array}
\end{pmatrix}.
\end{equation}
Note that for the considered system $n=3$, $p=2$, $l=2$, thus, it holds that $n<pl$.  

To the AR model we add the disturbance input $d(t) \in \mathbb{R}$ associated with the matrix
\begin{equation}
\nonumber
B_{d0}=
\begin{pmatrix}
0 \\
1
\end{pmatrix},
\end{equation}
for which it holds that  $\text{im}(\mathcal{C}(\hat{A},\hat{B}_d )\subseteq \text{im}(\mathcal{C}(\hat{A},\begin{pmatrix} \hat{B} & \hat{B}_1 \end{pmatrix} )$. See \eqref{eq:3:11} for the definitions of the corresponding matrices. In this case, the state $\chi(t)$ cannot span the whole space, and we can use persistently exciting input $\text{col}(u,w)$ to achieve Assumption~\ref{Assumption-1}.

The considered performance output $z(t)$ and  output channel available for control $y_c(t)$ are defined by the following matrices
\begin{equation}
\nonumber
\hat{C}_1:= \begin{pmatrix} 1 & 0_{1 \times 7} & -1 & 0 \end{pmatrix}
,\quad D_1=0,\quad 
E=0_{1 \times 2}, 
\end{equation}
\begin{equation}
\nonumber
 \quad \hat{C}=\begin{pmatrix} I_2 & 0_{2\times 8}  \end{pmatrix},
\end{equation}
 thus $z(t)=y_1(t-1)-w(t-1)$ and $y_c(t)=y(t-1)$, where $y=\text{col}(y_1,y_2)$.

This system model is used to generate exact and noisy input-output data.
The disturbance bound and the method for generating the noisy input-output data are the same as in \cite{c6}.
Samples of input $\text{col}(u,w)$ and initial condition $\chi(0) \in \mathcal{C}(\hat{A},\begin{pmatrix} \hat{B} & \hat{B}_1 \end{pmatrix}) $ are generated using a Gaussian distribution with zero mean and unit variance. Samples of disturbance $d$ are generated in the same manner, but with standard deviation $ \sigma \in \{0, 0.01, 0.05, 0.1, 0.2\}$ for noiseless case and noisy cases, respectively.
The disturbance bounds have the following form
\begin{equation}
W_dW_d^\top\preceq 1.35 N \sigma^2 I, \nonumber
\end{equation}
where the matrix $W_d$ is defined as in \eqref{NoiseData}, and $N=1000$. 
Assumptions~\ref{Assumption-2}, \ref{Assumption1} and \ref{Assumption-3} regarding the data matrices and the disturbance bound were verified. 
For controller synthesis we consider minimizations of $H_\infty$ performance $\gamma$ (as defined in Remark~\ref{RemarkPrvi1}) and $H_2$ performance $\mu$.
To solve the synthesis LMIs along with minimization of  $\gamma$ and $\mu$ we used Yalmip \cite{c22} environment in MATLAB, with Mosek as a LMI solver.

\begin{figure}[t!]
\centering
\vspace{-0.5cm}
\includegraphics{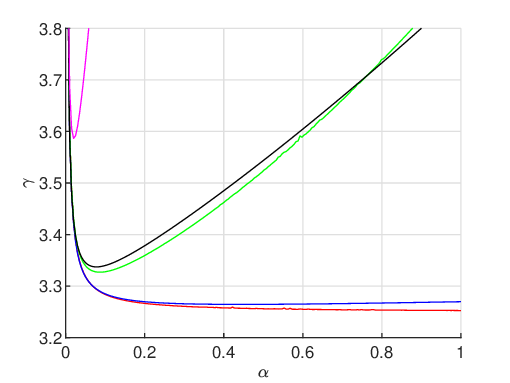}
\caption{$H_\infty$ performance $\gamma$ of the closed-loop system as a function of the
parameter $\alpha$ with $\sigma \in \{0, 0.01, 0.05, 0.1, 0.2\}$  in red, blue, green, black and magenta color, respectively.}
\label{fig-1}
\end{figure} 
 \begin{figure}[t!]
\centering
\includegraphics{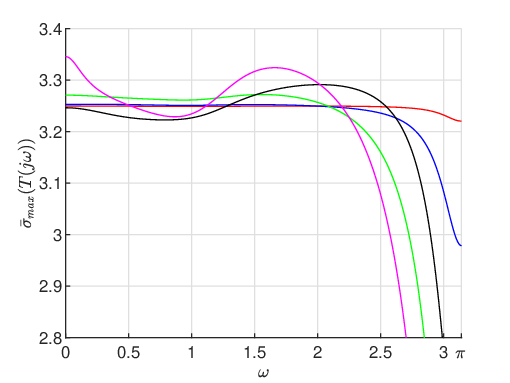}
\caption{
Frequency response of the closed-loop system with $\sigma \in \{0, 0.01,  0.05, 0.1, 0.2\}$  in red, blue, green, black and magenta color, respectively. }
\label{fig-2}
\end{figure}

The line search procedure for the $H_\infty$ controller synthesis is illustrated in the Fig.~\ref{fig-1}, where $H_\infty$ performance $\gamma$ is a function of the parameter $\alpha$.
 The $H_\infty$ performance $\gamma$ results are presented in Fig.~\ref{fig-2}, where by the term frequency response we refer to the largest singular value of the
corresponding transfer function on a particular
frequency $\omega $.  
Numerical results for $H_\infty$ and $H_2$ performance are presented in Table~\ref{Tablica1}, 
 where $\gamma$ and $\mu$ bounds
are $H_\infty$ and $H_2$ performance bounds, respectively, which are associated with the set of systems consistent with the recorded data, while $\gamma$ and  $\mu$ are $H_\infty$ and $H_2$ performances of the actual closed-loop system, respectively.

Note that closed-loop performances using model-based and data-based controllers obtained with $\sigma=0$ are the same. The performance of the closed-loop systems gets worse with an increase of noise in the recorded data,  which is to be expected since set of systems consistent with data increases.

\begin{table}[h!]
\renewcommand{\arraystretch}{1.3}
\caption{Numerical results for $H_\infty$ and $H_2$ control.}
\begin{center}
\label{Tablica1}
\begin{tabular}{ |c|c|c|c|c| } 
\hline
$\sigma$  & $\gamma$ & $\gamma$ bound  & $\mu$ & $\mu$ bound\\
\hline
0  & 3.25 & 3.25 & 3.00 & 3.00\\ 
0.01  & 3.25 & 3.26  & 3.00 & 3.02\\ 
0.05   & 3.27 & 3.33 & 3.00 & 3.11 \\ 
0.1   & 3.29  & 3.34 & 3.00  & 3.16	 \\
0.2  & 3.35  & 3.59 & 3.00 & 3.48 \\
\hline
\end{tabular}
\end{center}
\end{table}

\subsection{Robust $H_\infty$ control of an active suspension}
\label{Example2}
Consider a continuous-time state-space model of a quarter-car model of the active suspension system from \cite{MatlabRef} presented on Fig~\ref{fig40}.
\begin{figure}[h!]
\centering
\includegraphics{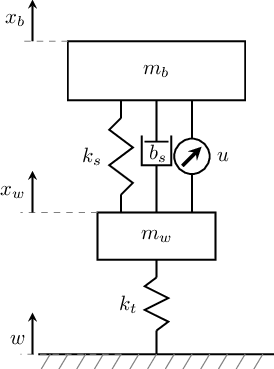}
\caption{
Active suspension system.
}
\label{fig40}
\end{figure}
The associated AR model is defined with the following matrices 
\begin{equation}
\nonumber
\underline{A} =
\begin{pmatrix}
1.9958 &	-0.0308 &	-1.0008 &	0.0310 \\
0.2834  &	1.5110 &	-0.0048 &	-0.8179
\end{pmatrix},
\end{equation}
\begin{equation}
\nonumber
\underline{B}_u =
10^{-3}
\begin{pmatrix}
0.1558 &	0.1455 \\
0.9140	& 0.8545
\end{pmatrix},
\end{equation}
\begin{equation}
\nonumber
\quad 
\underline{B}_w =
\begin{pmatrix}
0.0017 &	0.0034 \\
-0.1441 	& -0.1346
\end{pmatrix}, 
\end{equation}
\begin{equation}
\nonumber
\quad B_{u0}=\begin{pmatrix} 0 \\ 0 \end{pmatrix}, \quad B_{w0}=\begin{pmatrix} 0 \\ 0 \end{pmatrix},
\end{equation}
and signals $y$, $u$ and $w$, with $y=\text{col}(x_b,x_b-x_w)$ where $x_b \text{ [m]}$ and $x_w \text{ [m]}$ are the car body position and the wheel assembly position, respectively,
$u \text{ [kN]}$ is the force of active component of the suspension, and $w \text{ [m]}$ is the  the road disturbance.
These matrices are obtained following the same procedure as in the Section~\ref{Example1}, after applying zero-order hold method on the
 state-space matrices
\begin{equation}
\nonumber
A_m=\begin{pmatrix}
0 & 1 & 0 & 0\\
-\frac{k_s}{m_b} & -\frac{b_s}{m_b} & \frac{k_s}{m_b}  & \frac{b_s}{m_b} \\
0 & 0 & 0 & 1 \\
\frac{k_s}{m_w} & \frac{b_s}{m_w} & \frac{-k_s-k_t}{m_w} &   -\frac{b_s}{m_w}
\end{pmatrix},
\end{equation}
\begin{equation}
\nonumber
B_m=\begin{pmatrix}  0\\
  \frac{10^3}{m_b} \\
   0 \\
   -  \frac{10^3}{m_w} 
\end{pmatrix}, \quad 
B_{m1}= \begin{pmatrix} 0 \\
 0 \\
  0  \\
   k_t 
\end{pmatrix},
\end{equation}
\begin{equation}
\nonumber
C_{m}=\begin{pmatrix}
1 & 0 & 0 & 0 \\
1 & 0 & -1 & 0 
\end{pmatrix}, \quad  D_{m}=\begin{pmatrix} 0 \\ 0 \end{pmatrix}, \quad  D_{m1}=\begin{pmatrix} 0 \\ 0 \end{pmatrix}.
\end{equation}
with sampling time $T_s=0.01s$.
 These matrices are 
associated with the state-space vector $\text{col}(x_b, v_b, x_w, v_w)$, where $v_b \text{ [m/s]}$ and $v_w \text{ [m/s]}$ are  the car body and the wheel assembly velocity, respectively, $k_s=1.6\cdot 10^4 \frac{\text{N}}{\text{m}}$ is the spring stiffness, $m_b=300 \text{kg}$  is  the quarter-car body mass, $b_s=1000 \frac{\text{N}}{\text{m/s}}$ is the damping coefficient,  $m_w=60 \text{kg}$  is  the wheel assembly mass and $k_t=1.9\cdot 10^5 \frac{\text{N}}{\text{m}}$ is the compressibility of the pneumatic tire. 
Note that for the considered system $n=4$, $p=2$, $l=2$, thus, it holds that $n=pl$.

To the AR model we add the disturbance input $d(t) \in \mathbb{R}$ associated with matrix
\begin{equation}
\nonumber
B_{d0}=
\begin{pmatrix}
1 \\
1
\end{pmatrix},
\end{equation}
and 
we define performance output $z(t)$ and  output channel available for control $y_c(t)$
 using the following matrix
\begin{equation}
\nonumber
\hat{C}_1:= 
\begin{pmatrix}
I_2 & 0_{2\times 6}
\end{pmatrix}
,
\quad
D_1= 0_{2\times 1},
\quad 
E= 0_{2\times 1}, 
\end{equation}
\begin{equation}
\nonumber
\hat{C}=\begin{pmatrix} 0 & 1 & 0_{1\times 6}
\end{pmatrix},
\end{equation}
thus $z(t)=\text{col}(x_b(t-1),x_b(t-1)-x_w(t-1))$ and $y_c(t)=x_b(t-1)-x_w(t-1)$.

We consider the same problem setting as in the Section~\ref{Example1}, except
the samples of inputs $u$ and $w$  are randomly generated from the interval  $[-10,10]$ and $[-0.1,0.1]$, respectively, initial conditions are equal to zero, $\sigma \in \{0,0.001\}$ for noiseless and noisy case, respectively,  and $N=2000$.

As in the previous example, derived method recovers the
$H_\infty$ performance of model-based synthesis ($\gamma = 1.27$) for
the noiseless case ($\sigma = 0$). On the other hand, histogram
on Fig. \ref{fig44} illustrates performance robustness using $1000$
generated system trajectories for the noisy case ($\sigma = 0.001$).
Moreover, there were no inconsistency of $H_\infty$ performance of an actual closed-loop system and its bound obtained by synthesis procedure.
The presented results show that the derived method achieve
good robust performance results.

\section{CONCLUSIONS}
In this paper we have proposed non-conservative data-driven dynamic output-feedback controller synthesis methods for generic discrete-time LTI systems with the closed-loop performance criteria formalized in terms of dissipativity and $H_2$ performance. The data used for synthesis is the exact and noisy input-output data, and this complements existing static state-feedback controller
synthesis methods based on exact and noisy input-state data.
The presented numerical examples illustrates the effectiveness of the proposed methods.

 \begin{figure}[t!]
\centering
\includegraphics{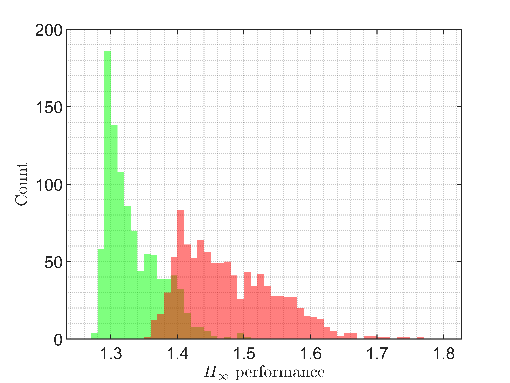}
\caption{
Green and red columns represent counts of $H_\infty$ performances of an actual closed-loop systems and their bounds obtained by synthesis procedure, respectively.}
\label{fig44}
\end{figure}

\section*{APPENDIX}

\subsection{Auxiliary definitions, abbreviations and relations}
\label{Appendix0}

The following definitions are used in the hypothesis of Theorem~\ref{Theorem1} and Theorem~\ref{Theorem2}, as well as in their proofs. 

Consider \eqref{eq:2:3} and define  
$\Lambda \in \mathbb{R}^{\tilde{p} \times \tilde{p}}$ as a positive definite matrix obtained by an arbitrary factorization
\begin{equation}
\label{ipakTreba3}
\begin{pmatrix}
\star \\
\end{pmatrix}
H
\begin{pmatrix}
I & \check{A}_s & \check{B}_{u0}  & \check{B}_{w0}  
\end{pmatrix}^\top=\Xi \Lambda \Xi^\top,
\end{equation}
where
\begin{equation}
\label{set_element}
\begin{pmatrix}
\check{A}_s & \check{B}_{u0}  & \check{B}_{w0}  
\end{pmatrix}^\top
 :=
  -\begin{pmatrix}
  H_{22} &  H_{23} & H_{24} \\
 H_{23}^\top & H_{33} & H_{34} \\
 H_{24}^\top & H_{34}^\top & H_{44} 
\end{pmatrix}^{-1}
\begin{pmatrix}
H_{12}^\top \\
H_{13}^\top  \\
H_{14}^\top
\end{pmatrix} .
\end{equation}
Note that according to the argumentation in \cite[expr. (3.4)]{c21},
 equation \eqref{set_element}  can be used to compute a
triple $(\bar{A}_s,B_{u0},B_{w0}) \in \Sigma$. This has been used in the proof of Theorems~\ref{Theorem1} and \ref{Theorem2}. 

Consider \eqref{set_element} and the matrix partition $X_s=\text{col}(X_{s1}, X_{s2})$ where $X_{s1}\in \mathbb{R}^{p\times \tilde{n}}$ with respect to the Assumption~\ref{Assumption-2}.
We make the following definitions
\begin{equation}
\nonumber
\check{A}:= X_s^\top \begin{pmatrix} \check{A}_s \\ J_A X_s \end{pmatrix}, \quad 
\check{B}:=X_s^\top \begin{pmatrix} \check{B}_{u0} \\ J_B \end{pmatrix},
\end{equation}
\begin{equation}
\nonumber
\check{B}_1:=X_s^\top \begin{pmatrix} \check{B}_{w0} \\ J_{B_1} \end{pmatrix}, \quad 
L :=
\begin{pmatrix}
L_1 \\ L_2
\end{pmatrix}, \quad
K:= \begin{pmatrix} 0 & K_{12} \\ I_p & K_{22} \end{pmatrix},
\end{equation}
\begin{equation}
\label{ipakTreba2}
\begin{pmatrix}
\bar{H}_{12}^\top \\
\bar{H}_{13}^\top \\
\bar{H}_{14}^\top
\end{pmatrix}:=
\begin{pmatrix}
H_{12}^\top & H_{22} & H_{23} & H_{24} \\
H_{13}^\top & H_{23}^\top & H_{33} & H_{34} \\
H_{14}^\top & H_{24}^\top & H_{34}^\top & H_{44} \\
\end{pmatrix}
\begin{pmatrix}
I \\ \check{A}_s^\top \\ \check{B}_{u0}^\top   \\ \check{B}_{w0}^\top 
\end{pmatrix},
\end{equation}
where $L_1^\top$ is an arbitrary full column matrix whose columns form a basis of $\text{ker}(X_{s1})$, $L_2=(X_{s1} X_{s1}^\top)^{-1} X_{s1}$, $K_{12}= L_1 X_{s2}^\top$ and $K_{22}=L_2 X_{s2}^\top $.

Note that matrices $L$ and $K$ defined above are solution of the matrix equation
\begin{equation}
\label{ipakTreba1}
L X_{s}^\top=K.
\end{equation}

\subsection{Synthesis certificates and controller realizations}
\label{AppendixA}
Consider the controller synthesis LMIs \eqref{eq:3:4} and corresponding closed loop system defined by \eqref{closed_loop} and \eqref{closed_loop_matrices}.
We call the matrix $P$ that satisfies these LMIs \emph{the synthesis certificate}. 
Here we present some properties of synthesis certificates $P$ related to its partition in conformity with the partition of the closed-loop state vector $\xi$ in \eqref{closed_loop}. These properties are instrumental for the contruction of controller synthesis procedure.  
\begin{lemma}
\label{lemma3}
Suppose the inequalities \eqref{eq:3:4} hold for the closed-loop system \eqref{closed_loop} with some controller state space realization matrices $(\bar{A}_c, \bar{B}_c, \bar{C}_c, \bar{D}_c)$ and the corresponding synthesis certificate \begin{small}$\bar{P}:=\begin{pmatrix} X & \bar{U} \\ \bar{U}^\top & \bar{M}_1 \end{pmatrix}$\end{small} where $X \in \mathbb{R}^{n \times n}$. Let $\tilde{U} \in \mathbb{R}^{n \times n}$ be an arbitrary given matrix. Then there exists a regular matrix $\tilde{L}\in \mathbb{R}^{n \times n} $ such that a controller space state realization matrices 
$(A_c, B_c, C_c, D_c):=(\tilde{L}\bar{A}_c \tilde{L}^{-1}, \tilde{L} \bar{B}_c, \bar{C}_c \tilde{L}^{-1}, \bar{D}_c)$ and the synthesis certificate \begin{small}$P:=\begin{pmatrix} X & \tilde{U} \\ \tilde{U}^\top & M_1 \end{pmatrix}$\end{small} 
satisfy \eqref{eq:3:4}.  
\end{lemma}
\begin{proof} 
Recall the matrix partition in \eqref{closed_loop_matrices} and note that we can always perturb $\bar{U}$ to obtain a regular matrix since \eqref{eq:3:4} is a strict inequality. 
Next, we apply the congruence transformation on \eqref{E1} and \eqref{E2} with respect to matrices $\text{diag}(I_{\tilde{n}}, \tilde{L}^{-1})$ and $\text{diag}(I_{\tilde{n}},\tilde{L}^{-1}, I_{m_w} )$, respectively.
Then by using the substitutions
$(\bar{A}_c, \bar{B}_c, \bar{C}_c, \bar{D}_c)=(\tilde{L}^{-1}A_c \tilde{L}, \tilde{L}^{-1} B_c, C_c \tilde{L}, D_c)$, $\tilde{U}=\bar{U}\tilde{L}^{-1}$ and $M_1=\tilde{L}^{-\top}\bar{M}_1 \tilde{L}^{-1}$ we obtain a new set of LMIs in the form of \eqref{eq:3:4}
with the controller space state realization matrices $(A_c, B_c, C_c, D_c)$ and
the synthesis certificate \begin{small}$P=\begin{pmatrix} X & \tilde{U} \\ \tilde{U}^\top & M_1 \end{pmatrix}$\end{small}.
Therefore, we can conclude that by appropriate choice of $\tilde{L}$ we can render arbitrary $\tilde{U}$ from some given $\bar{U}$.      
\end{proof}
Note that the Lemma~\ref{lemma3} also holds if we consider \eqref{eq:3:4_H2} instead of  \eqref{eq:3:4}. We can prove this by following the proof of the Lemma~\ref{lemma3}
and substituting \eqref{E1}, \eqref{E2}, $\text{diag}(I_{\tilde{n}}, \tilde{L}^{-1})$, $\text{diag}(I_{\tilde{n}},\tilde{L}^{-1}, I_{m_w} )$, $\tilde{U}=\bar{U}\tilde{L}^{-1}$, $M_1=\tilde{L}^{-\top}\bar{M}_1 \tilde{L}^{-1}$ with  \eqref{E2_H2}, \eqref{E3_H2}, $\text{diag}(I_{m_w+\tilde{n}}, \tilde{L}^\top)$, $\text{diag}(I_{p_z+\tilde{n}},\tilde{L}^{\top})$, $\tilde{U}=\bar{U}\tilde{L}^{\top}$, $M_1=\tilde{L}\bar{M}_1 \tilde{L}^{\top}$, respectively.

\addtolength{\textheight}{-12cm}  
\balance

\bibliographystyle{IEEEtran}
\bibliography{references}

\end{document}